\documentclass[12pt]{amsart}

\usepackage{amssymb}
\usepackage{hyperref} 
\usepackage{enumitem}

\usepackage{amsmath}
\usepackage{latexsym}
\usepackage{extarrows}
\usepackage{enumerate}
\usepackage{txfonts}
\usepackage{mathtools}
\usepackage{bm}
\usepackage{tikz-cd}
\usepackage{xcolor}

\makeatletter
\@namedef{subjclassname@2020}{%
  \textup{2020} Mathematics Subject Classification}
\makeatother

\usepackage[T1]{fontenc}

\theoremstyle{plain}

\newtheorem{theorem}{Theorem}[section]
\newtheorem{proposition}[theorem]{Proposition}

\newtheorem{lemma}[theorem]{Lemma}
\newtheorem{corollary}[theorem]{Corollary}

\theoremstyle{definition}

\newtheorem{definition}[theorem]{Definition}

\newtheorem{example}[theorem]{Example}

\newcommand\E{\mathbb{E}}
\newcommand\Z{\mathbb{Z}}
\newcommand\R{\mathbb{R}}
\newcommand\Q{\mathbb{Q}}
\newcommand\T{\mathbb{T}}

\newcommand\eps{\varepsilon}

\DeclareMathOperator{\Id}{Id}
\DeclareMathOperator{\Stab}{Stab}
\DeclareMathOperator{\diam}{diam}

\begin{document}


\baselineskip=17pt


\title[]{Local rigidity of self-joinings and factors of pro-nilsystems}

\author[P. Eeckhaut]{Pauwel Van Den Eeckhaut}
\address{Institute of Mathematics \\ University of Bonn \\ 53113 Bonn, Germany}
\email{s97pvand@uni-bonn.de}

\author[A. Jamneshan]{Asgar Jamneshan}
\address{Institute of Mathematics \\ University of Bonn \\ 53113 Bonn, Germany}
\email{ajamnesh@math.uni-bonn.de}

\date{\today}

\begin{abstract}  
It is an immediate consequence of the ergodic structure theorem of Host and Kra that every factor of an ergodic $k$-step pro-nilsystem is again an ergodic $k$-step pro-nilsystem. It has remained open whether this fact can be proved independently of the structure theorem itself. In this note, we give such a proof. The key new ingredient is a local rigidity theorem for nilsystems: any ergodic self-joining sufficiently close to the diagonal joining is necessarily the graph joining of an automorphism. This rigidity result may be of independent interest. As an application, our proof of the factor-closure of pro-nilsystems combined with a result of Tao yields a new proof of the ergodic structure theorem of Host and Kra from the combinatorial inverse theorem of Green, Tao, and Ziegler for the Gowers norms on cyclic groups. We can also use our methods to establish an independent proof of the factor-closure property of topological pro-nilsystems, a fact that can also be derived from the topological structure theorem of Host, Kra, and Maass.
\end{abstract}



\maketitle

\section{Introduction}

A $k$-step nilsystem is a measure-preserving dynamical system $(X,\mu,T)$ such that
$X = G/\Gamma$ is a $k$-step nilmanifold, $\mu$ is its Haar probability measure, and
$T \colon X \to X$ is the nilrotation $T(x) = \tau \cdot x$ for some fixed element $\tau \in G$.
A $k$-step pro-nilsystem is a measure-preserving dynamical system that is the inverse
limit, in the measure-theoretic sense, of a countable inverse system of $k$-step
nilsystems. All background and facts on nilsystems used in
this paper can be found in \cite[Part~3]{hk-book}, whose notation we adopt throughout.

Pro-nilsystems occupy a central place in the study of multiple ergodic
averages introduced by
Furstenberg \cite{furstenberg1977ergodic} in his ergodic-theoretic proof of
Szemer\'edi's theorem. 
Host and Kra \cite{host2005nonconventional} introduced a theory of cubical
structures and associated seminorms, used these to define systems of order $k$,
and established a structure theorem identifying ergodic systems of order $k$
with $k$-step pro-nilsystems (see Theorem \ref{hkz-thm} below). As a consequence, they identified
characteristic factors for multiple ergodic averages and proved their $L^2$-convergence. 
Independently, Ziegler \cite{ziegler2007universal} also proved the
$L^2$-convergence of these averages by showing that the universal
characteristic factors she introduced are pro-nilsystems, using a
different construction of these factors. The relation
between the Host--Kra factors of order $k$ and Ziegler's $k$th universal characteristic factors was clarified by Leibman
\cite{leibman-factor}, who showed that the two descriptions agree.

We briefly recall the seminorm formulation of the Host--Kra factors of order $k$. Let
$(X,\mu,T)$ be a measure-preserving dynamical system and let $f \in L^\infty(\mu)$. The
Gowers--Host--Kra seminorms are defined recursively by
\[
\|f\|_{U^1(X)} := \left| \int_X f\,d\mu \right|
\]
and, for $k \geq 1$,
\[
\|f\|_{U^{k+1}(X)}^{2^{k+1}}
:=
\lim_{N\to\infty} \frac{1}{N}\sum_{n=1}^N
\bigl\| \overline{f}\, T^n f \bigr\|_{U^k(X)}^{2^k}.
\]
The system $(X,\mu,T)$ is said to be of order $k$ if
\[
\|f\|_{U^{k+1}(X)} = 0
\quad\Longrightarrow\quad
f=0 \ \ \mu\text{-almost everywhere}
\]
for every $f \in L^\infty(\mu)$.

\begin{theorem}[Host--Kra structure theorem]\label{hkz-thm}
Let $k \geq 1$. An ergodic measure-preserving dynamical system is of order $k$
if and only if it is measure-theoretically isomorphic to a $k$-step
pro-nilsystem.
\end{theorem}

A natural question is whether the class of $k$-step pro-nilsystems is closed
under taking factors\footnote{Recall that a factor of a measure-preserving system $(X,\mu,T)$ is a
measure-preserving system $(Y,\nu,S)$ for which there exists a
measure-preserving map $\pi\colon X \to Y$ such that $\pi \circ T = S \circ \pi$ $\mu$-almost surely.}. Since any factor of a measure-preserving dynamical system of order $k$ is again a
measure-preserving dynamical system of order $k$ (see
\cite[Proposition~17, Chapter~9]{hk-book}), it follows immediately from
Theorem~\ref{hkz-thm} that:

\begin{theorem}\label{thm-main}
Any factor of an ergodic $k$-step pro-nilsystem is itself a $k$-step
pro-nilsystem.
\end{theorem}

It has been an open problem in the subject whether
Theorem~\ref{thm-main} admits an independent proof. As Tao writes in
\cite{tao-weak}:\footnote{Here, Tao's Theorem~4 corresponds to
Theorem~\ref{thm-main}, while his Theorem~2 corresponds to
Theorem~\ref{hkz-thm}.}
\begin{quote}
Theorem 4 is, in principle, purely a fact about nilsystems, and should have an
independent proof, but this is not known; the only known proofs go through the
full machinery needed to prove Theorem 2.
\end{quote}
Similarly, Host and Kra write in \cite[p.~221]{hk-book}:
\begin{quote}
However, unfortunately the only proof that we know for this result relies on the ergodic Structure Theorem stated in Chapter 16,
and at this point we are not able to prove this in a more elementary manner.
\end{quote}

In this paper, we give such an independent proof. Our argument avoids the
machinery behind the structure theorem entirely and relies only on intrinsic
properties of nilsystems; see Section \ref{sec:proof} for a discussion of the innovations in our proof.
The fact that a factor of a nilsystem is again a
nilsystem was established earlier by Parry \cite{parry}; see also
\cite{lesigne,leibman}.

We obtain the following interesting application from our proof of Theorem \ref{thm-main}. 
In \cite{tao-weak}, Tao showed that the inverse theorem for the Gowers norms on
cyclic groups, due to Green, Tao, and Ziegler \cite{gtz},
implies that every measure-preserving dynamical system of order $k$ is a factor
of a $k$-step pro-nilsystem. Combined with our proof of
Theorem~\ref{thm-main}, this gives a new proof of Theorem~\ref{hkz-thm}, and
clarifies the logical relationship between the combinatorial and
ergodic-theoretic inverse theorems.

Finally, we record the following topological dynamical analogue of Theorem \ref{thm-main}.

\begin{theorem} \label{thrm:topological_version_main_thrm}
    Every factor of a transitive topological $k$-step pro-nilsystem is again a topological $k$-step pro-nilsystem. 
\end{theorem}

This result is a consequence of the topological structure theorem of Host,
Kra, and Maass; see \cite[Theorem~1.2]{host2010nilsequences}. In their
terminology, a transitive topological \(k\)-step pro-nilsystem is a system of
order \(k\), and the class of systems of order \(k\) is closed under taking
factors. Thus a topological factor of a transitive \(k\)-step pro-nilsystem is
again a system of order \(k\), and hence an inverse limit of \(k\)-step
nilsystems. In Section~\ref{sec:topological}, we instead give a proof that relies only on
Theorem~\ref{thm-main}, whose proof is intrinsic to nilsystems and does not
invoke any structure theorem.

\subsection{Outline of the proof of Theorem \ref{thm-main}}\label{sec:proof}

Conceptually, our proof of Theorem~\ref{thm-main} has two main steps. First, we reduce the theorem to the case in which the given pro-nilsystem is a skew-product extension of the factor by a compact abelian group. We then prove this special case.

Let $X$ be an ergodic $k$-step pro-nilsystem, and let $\pi\colon X \to Y$ be a factor map. Proposition~\ref{prop:inverse_limit_extension} allows us to pass the tower of factors of each finite-stage nilsystem factor of $X$ to the inverse limit, thereby obtaining a tower
\[
X = Z_k \to Z_{k-1} \to \dots \to Z_1 \to Z_0 = \{*\},
\]
where each map $Z_j \to Z_{j-1}$ is an extension by a compact abelian group.

For each $0 \leq j \leq k$, we then introduce an intermediate factor $W_j$, defined as the smallest factor of $X$ of which both $Z_j$ and $Y$ are factors. In this way we obtain a tower
\[
X = W_k \to W_{k-1} \to \dots \to W_1 \to W_0 = Y.
\]
We prove by downward induction on $j$ that each $W_j$ is a $k$-step pro-nilsystem. Since $W_k = X$, the base case is immediate. For the induction step, the extension $Z_j \to Z_{j-1}$ induces an extension $W_j \to W_{j-1}$ as well. Thus the induction step is reduced precisely to the compact abelian extension case. Once this special case is established, the induction shows that $W_0 = Y$ is a $k$-step pro-nilsystem.

It therefore remains to prove the compact abelian extension case, namely Theorem~\ref{thrm:special_case}. Suppose that $X = \varprojlim X_i$ is an ergodic $k$-step pro-nilsystem, and that $X$ is a skew-product extension of $Y$ by a compact abelian group $K$. Writing $p_i\colon X \to X_i$ for the factor maps, the goal is to use the finite-stage nilsystems $X_i$ to construct an inverse limit presentation of $Y$.

The main obstruction is that the given inverse limit presentation of $X$ need not be compatible with the skew-product structure on $X$. More precisely, the vertical rotations $V_u$, for $u \in K$, do not in general descend to automorphisms of the systems $X_i$. Equivalently, the factors $X_i$ need not be invariant under all vertical rotations.

To handle this obstruction, for each $i$ and each $u \in K$, we consider the self-joining
\[
\lambda_u^{(i)} := (p_i, p_i \circ V_u)_* \mu \qquad \text{on } X_i \times X_i.
\]
If $\lambda_u^{(i)}$ is the graph joining of an automorphism of $X_i$, then $V_u$ does descend to an automorphism of $X_i$.

The crucial input here is the following local rigidity statement, where we denote by $J_e(X,X)$ the set of ergodic self-joinings of a nilsystem $X$. 

\begin{proposition}\label{prop:neighbourhood_diag_joining_graph}
Let $(X,\mu,T)$ be an ergodic nilsystem. Let $\mu_\Delta := (\mathrm{Id}_X,\mathrm{Id}_X)_*\mu$ be the diagonal joining. Then there exists $\delta>0$ such that, for every
$\lambda\in J_e(X,X)$, if\footnote{We view $J_e(X,X)$ as a subspace of the space $\mathrm{Pr}(X\times X)$  of Borel probability measures on the compact Hausdorff space $X\times X$ which, by the Riesz representation theorem, is naturally identified with a weak-* compact subset of the dual of $C(X\times X)$. 
Since $X\times X$ is metrizable, this topology is metrizable on $\mathrm{Pr}(X\times X)$, and we let $d$ be any metric inducing it.} $d(\lambda,\mu_\Delta)<\delta$, then $\lambda$ is the graph joining of an automorphism\footnote{By an automorphism of $(X,\mu,T)$ we mean an automorphism in the measurable sense, that is, an invertible measurable factor map from $(X,\mu,T)$ to itself, defined up to almost-everywhere equality.} of $(X,\mu,T)$. 

Equivalently, every ergodic self-joining of $(X,\mu,T)$ which is not the
graph joining of an automorphism lies at distance at least $\delta$ from the
diagonal joining. 
\end{proposition}

The proof of this proposition is one of the main new ideas of the paper. 
The key fact is that an ergodic self-joining of a nilsystem must be the Haar measure of a subnilmanifold of the ambient product nilmanifold; see Proposition~\ref{prop:joining_nilsystems_is_nilsystem}.
Thus the problem becomes geometric.

The relevant geometric input is a no small subnilmanifolds lemma and its averaged form; see Lemma \ref{lem:no_small_subnilmanifolds} and Lemma \ref{lem:no_small_subnilmanifolds_averaged}.
Roughly speaking, Lemma \ref{lem:no_small_subnilmanifolds} says that nilmanifolds cannot contain arbitrarily small non-trivial subnilmanifolds, while Lemma \ref{lem:no_small_subnilmanifolds_averaged} gives a stronger version of this statement in terms of the average distance between two points in such a subnilmanifold. 
These results may therefore be viewed as analogues of the no small subgroups property of Lie groups, in the setting of nilmanifolds.

Suppose now that an ergodic self-joining of a nilsystem is sufficiently close to the diagonal joining.
This implies that the two coordinates of points in its support are close to each other on average.
If the joining were not a graph joining, then its supporting subnilmanifold would have a non-trivial coordinate fiber.
By the homogeneous structure of this subnilmanifold, this non-triviality cannot occur only at a single point: it persists throughout the support. 
But Lemma \ref{lem:no_small_subnilmanifolds_averaged} then gives a lower bound on the average distance between the two coordinates of points in the support.
This contradicts the assumed closeness to the diagonal joining.
Hence the joining must be the graph joining of an automorphism.

Using Proposition~\ref{prop:neighbourhood_diag_joining_graph}, we show for each $i$ that $\lambda_u^{(i)}$ is the graph joining of an automorphism for every $u$ in some open subgroup $H_i \leq K$. Choosing finitely many representatives for the quotient $K/H_i$, we then enlarge the system $X_i$ to a finite self-joining $\widetilde X_i$ which is invariant under all vertical rotations. For each $i$, let $\widetilde W_i$ be the largest common factor of $\widetilde X_i$ and $Y$. Since $\widetilde X_i$ is an ergodic finite self-joining of the nilsystem $X_i$, it is itself an ergodic $k$-step nilsystem, and hence so is its factor $\widetilde W_i$.

The invariance of the factors $\widetilde X_i$ under vertical rotations makes it possible to average over the group $K$, and from this we show that the factors $\widetilde W_i$ generate $Y$. It follows that $Y$ is the inverse limit of the systems $\widetilde W_i$, and therefore $Y$ is a $k$-step pro-nilsystem.

\subsection*{Acknowledgments}
The results of this paper comprise the Master's thesis of the first author.
We are grateful to Bryna Kra for several helpful comments and suggestions on
an earlier draft of the manuscript. We are also grateful to Sebasti\'an Donoso
for asking whether our methods could be applied in the setting of topological
dynamical systems. This question prompted us to prove
Theorem~\ref{thrm:topological_version_main_thrm}. 
AJ was funded by the German Research Foundation under Germany's Excellence Strategy -- EXC-2047 -- 390685813 and its Heisenberg Programme --  547294463.

\section{No small subnilmanifolds}\label{sec:nosmall}

Subnilmanifolds play an important role in the proof of Theorem~\ref{thm-main}. In this section, we introduce subnilmanifolds and prove a key lemma that will be used later.

\begin{definition}[Subnilmanifold]
    Let $X = G/\Gamma$ be a nilmanifold. A subset $Y \subseteq X$ is called a subnilmanifold of $X$ if $Y$ is closed in $X$ and there exist a closed subgroup $H \leq G$ and a point $x \in X$ such that
    \[
    Y = H \cdot x.
    \]
\end{definition}

Suppose that $Y = H \cdot x$ is a subnilmanifold of $X = G/\Gamma$. Then $H$ acts smoothly and transitively on $Y$ by left translations, and hence
\[
Y \cong H/\Stab_H(x).
\]
If $x = a\Gamma$ for some $a \in G$, then
\[
\Stab_H(x) = H \cap a\Gamma a^{-1},
\]
so that
\[
Y \cong H/(H \cap a\Gamma a^{-1}).
\]
Since $H$ is a closed subgroup of the nilpotent Lie group $G$, it is again a nilpotent Lie group. Moreover, $H \cap a\Gamma a^{-1}$ is discrete in $H$ because $\Gamma$ is discrete in $G$, and it is cocompact because the quotient is identified with the compact space $Y$. Thus $Y$ itself naturally has the structure of a nilmanifold.

A particularly important class of subnilmanifolds consists of those of the form $L \cdot e_X$, where $L$ is a rational subgroup of $G$, in the sense of \cite[Chapter~10, Lemma~14]{hk-book}; equivalently, $L \cdot e_X$ is closed in $X$. These are precisely the subnilmanifolds of $X$ containing the base point $e_X$.

Every subnilmanifold is a translate of one containing the base point. Indeed, if
\[
Y = H \cdot x \qquad \text{with } x = a\Gamma,
\]
and we set $L = a^{-1}Ha$, then
\[
Y = H \cdot (a\Gamma) = (Ha)\cdot e_X = a \cdot (L \cdot e_X).
\]
Since left translation by $a$ is a homeomorphism, $L \cdot e_X$ is closed whenever $Y$ is closed. Hence $L$ is rational.

One characteristic property of Lie groups is that they have no small subgroups: every Lie group admits an identity neighbourhood containing no non-trivial subgroup. The following lemma may be viewed as a nilmanifold analogue of this fact. 

\begin{lemma}[No small subnilmanifolds] \label{lem:no_small_subnilmanifolds}
    Let $X = G/\Gamma$ be a $k$-step nilmanifold, and let $d_X$ be a compatible metric on $X$. Then there exists $\eps > 0$ such that any subnilmanifold $Y \subseteq X$ with
    \[
    \diam_X(Y) < \eps
    \]
    consists of a single point.
\end{lemma}

This lemma gives a uniform lower bound on the diameter of subnilmanifolds which do not consist of a single point.
We will refer to such subnilmanifolds as non-trivial. 
For the proof of Proposition \ref{prop:neighbourhood_diag_joining_graph}, however, we will need the following stronger averaged form of this statement. 
We will prove this averaged form directly, and then recover Lemma \ref{lem:no_small_subnilmanifolds} as a consequence. 

\begin{lemma}[No small subnilmanifolds, averaged form] \label{lem:no_small_subnilmanifolds_averaged}
    Let $X = G/\Gamma$ be a $k$-step nilmanifold, and let $d_X$ be a compatible metric on $X$. 
    Then there exists $c > 0$ such that for every non-trivial subnilmanifold $Y \subseteq X$, 
    \[
    \int_{Y \times Y} d_X(x,y) \,d m_Y(x) \,d m_Y(y) \geq c,
    \]
    where $m_Y$ denotes the Haar probability measure of $Y$.
\end{lemma}

Indeed, Lemma \ref{lem:no_small_subnilmanifolds_averaged} implies Lemma \ref{lem:no_small_subnilmanifolds}.
For every subnilmanifold $Y \subseteq X$ with Haar probability measure $m_Y$, 
\[
\int_{Y \times Y} d_X(x,y) \,d m_Y(x) \,d m_Y(y) \leq \diam_X(Y).
\]
Thus if $c > 0$ is the constant from Lemma \ref{lem:no_small_subnilmanifolds_averaged}, then every non-trivial subnilmanifold $Y \subseteq X$ satisfies
\[
\diam_X(Y) \geq c.
\]
Consequently, any subnilmanifold $Y \subseteq X$ with $\diam_X(Y) < c$ must be a singleton. 

The proof of Lemma \ref{lem:no_small_subnilmanifolds_averaged} will use the following simple consequence of the no small subgroups property for compact abelian Lie groups acting freely on compact metric spaces. 

\begin{lemma} \label{lem:no_small_subgroups_averaged}
    Let $K$ be a compact abelian Lie group acting freely and continuously on a compact metric space $X$ with metric $d_X$.
    Denote the action by 
    \[
    K \times X \to X, \qquad (u, x) \mapsto u \cdot  x.
    \]
    Then there exists $c > 0$ such that for every $x \in X$, and every non-trivial closed subgroup $L \leq K$, 
    \[
    \int_{L \times L} d_X(u \cdot x, v \cdot x) \,d m_L(u) \,d m_L(v) \geq c,
    \]
    where $m_L$ denotes the Haar probability measure of $L$.
\end{lemma}

\begin{proof}
    Since $K$ is a Lie group, it admits an identity neighbourhood $U$ containing no non-trivial subgroup. 
    Choose an identity neighbourhood $V \subseteq K$ such that 
    \[
    V V^{-1} \subseteq U. 
    \]
    Since the $K$-action is free, the continuous map
    \[
    (u,x) \mapsto d_X(u \cdot x, x)
    \]
    is strictly positive on the compact set $(K \setminus V) \times X$.
    Hence there exists $c > 0$ such that 
    \[
    d_X (u \cdot x, x) \geq 2 c
    \]
    for all $u \in K \setminus V$ and all $x \in X$.

    Now let $L \leq K$ be a non-trivial closed subgroup.
    We claim that 
    \[
    m_L(L \cap V) \leq 1/2.
    \]
    Indeed, if $m_L(L \cap V) > 1/2$, then 
    \[
    (L \cap V) (L \cap V)^{-1} = L,
    \]
    because any Borel subset of a compact group with Haar measure $> 1/2$ has difference set equal to the whole group. 
    This would imply that 
    \[
    L \subseteq V V^{-1} \subseteq U,
    \]
    contradicting the fact that $L$ is non-trivial.

    Therefore, for any $x \in X$, using invariance of the Haar measure $m_L$, we get that  
    \begin{align*}
        &\int_{L \times L} d_X(u \cdot x, v \cdot x) \,d m_L(u) \,d m_L(v) \\
        &= \int_L \left(
        \int_L d_X(u \cdot (v \cdot x), v \cdot x) \,d m_L(u) \right) \,d m_L(v) \\
        &\geq \int_L 2c \ m_L(L \setminus V) \,d m_L(v) \\
        &\geq c.
    \end{align*} 
\end{proof}

We now prove the averaged no small subnilmanifolds lemma.

\begin{proof}[Proof of Lemma \ref{lem:no_small_subnilmanifolds_averaged}]
    We argue by induction on the step $k$ of the nilmanifold $X$.

    First suppose that $k = 1$. 
    Then $X = G/\Gamma$ is a compact abelian Lie group.
    Applying Lemma \ref{lem:no_small_subgroups_averaged} to the translation action of $X$ on itself, we obtain $c > 0$ such that for every $a \in X$, and every non-trivial closed subgroup $L \leq X$, 
    \[
    \int_{L \times L} d_X(u \cdot a, v \cdot a) \,d m_L(u) \,d m_L(v) \geq c. 
    \]
    Let $Y \subseteq X$ be a non-trivial subnilmanifold.
    Then 
    \[
    Y = H \cdot a = \pi(H) \cdot a 
    \]
    for some closed subgroup $H \leq G$ and some $a \in X$. 
    Since $Y$ is closed, we equivalently have that $Y = L\cdot a$ where $L := \overline{\pi(H)} \leq X$.
    The closed subgroup $L$ is non-trivial, since $Y$ is non-trivial. 
    Moreover, the Haar measure $m_Y$ on $Y$ is the pushforward of $m_L$ under the map $u \mapsto u\cdot a$.
    Hence 
    \[
    \int_{Y \times Y} d_X(x,y) \,d m_Y(x) \,d m_Y(y)
    = \int_{L \times L} d_X(u \cdot a, v \cdot a) \,d m_L(u) \,d m_L(v) \geq c.
    \]
    This proves the $1$-step case.

    Now suppose that $k \geq 2$ and that the statement has been proved for $(k-1)$-step nilmanifolds.
    Let $G_k$ denote the last non-trivial term of the lower central series of $G$, and define
    \[
    Z := G/(G_k \Gamma) 
    \cong (G/G_k)/((G_k\Gamma)/G_k) 
    \cong G_k \backslash X.
    \]
    Then $Z$ is a $(k-1)$-step nilmanifold.
    Let $p \colon X \to Z$ be the natural projection, and choose a compatible metric $d_Z$ on $Z$. 
    We may assume without loss of generality that $d_Z \leq 1$.
    
    By the induction hypothesis, there exists $c' >0$ such that for every non-trivial subnilmanifold $W \subseteq Z$,
    \[
    \int_{W \times W} d_Z(z,w) \,d m_W(z) \,d m_W(w) \geq c'. 
    \]
    By compactness, $p$ is uniformly continuous, so there exists $\delta > 0$ such that 
    \[
    d_X(x,y) < \delta \implies d_Z(p(x), p(y)) < c'/2 
    \]
    for all $x,y \in X$.

    Next, let
    \[
    K := G_k/(G_k \cap \Gamma).
    \]
    Since $G_k$ is central and a rational subgroup of $G$, $K$ is a compact abelian Lie group. 
    It acts freely and continuously on $X$ by left translations, and the orbits of this action are precisely the fibers of $p$.

    Applying Lemma \ref{lem:no_small_subgroups_averaged} to this $K$-action, we obtain $c'' > 0$ such that for every $a \in X$, and every non-trivial closed subgroup $L \leq K$, 
    \[
    \int_{L \times L} d_X(u \cdot a, v \cdot a) \,d m_L(u) \,d m_L(v) \geq c''.
    \]

    We claim that 
    \[
    c := \min\{\delta c'/2, c''\} 
    \]
    satisfies the desired condition.

    Let $Y \subseteq X$ be a non-trivial subnilmanifold.
    Write 
    \[
    Y = H \cdot a
    \]
    for some closed subgroup $H \leq G$ and some $a \in X$.
    We distinguish two cases: either $p(Y)$ is also non-trivial, or $p(Y)$ is a single point.
    
    Suppose first that $p(Y)$ is non-trivial.  
    Let $H'$ be the closure of the image of $H$ under the quotient homomorphism $G \to G/G_k$.
    Since $Y$ is compact and $p$ is continuous, $p(Y)$ is closed in $Z$.
    Therefore under the identification $Z \cong (G/G_k)/((G_k\Gamma)/G_k)$, we see that 
    \[
    p(Y) = H' \cdot p(a),
    \]
    so $p(Y)$ is a subnilmanifold of $Z$.
    We claim that 
    \[
    p_* m_Y = m_{p(Y)},
    \]
    where $m_Y$ and $m_{p(Y)}$ denote the Haar measures of $Y$ and $p(Y)$ respectively. 
    Indeed, $p_* m_Y$ is a Borel probability measure on $p(Y)$, which is invariant under the translations induced by the image of $H$ in $G/G_k$, and hence under the translations by all elements of $H'$.
    The claim then follows by uniqueness of the Haar measure. 

    By our choice of $c'$, we now have that 
    \[
    c' \leq \int_{p(Y) \times p(Y)} d_Z(z,w) \,d m_{p(Y)}(z) \,d m_{p(Y)}(w). 
    \]
    Using $p_* m_Y = m_{p(Y)}$, this becomes 
    \[
    c' \leq \int_{Y \times Y} d_Z(p(x),p(y)) \,d m_Y(x) \,d m_Y(y).
    \]
    We split the integral according to whether $d_X(x,y) < \delta$ or $d_X(x,y) \geq \delta$.
    On the first region, the integrand is bounded by $c'/2$.
    On the second region, since $d_Z \leq 1$ and $d_X(x,y) \geq \delta$, we have 
    \[
    d_Z(p(x), p(y)) \leq 1 \leq \frac{1}{\delta} d_X(x,y).
    \]
    Therefore 
    \[
    c' \leq c'/2 + \frac{1}{\delta} \int_{Y \times Y} d_X(x,y) \,d m_Y(x) \,d m_Y(y).
    \]
    Hence 
    \[
    \int_{Y \times Y} d_X(x,y) \,d m_Y(x) \,d m_Y(y) \geq \delta c'/2 \geq c.
    \]

    It remains to consider the case where $p(Y)$ is a singleton.
    Then $Y$ is contained in a single fiber of $p$, equivalently $Y$ is contained in a single orbit of $K$.
    Since $a \in Y$, this implies that 
    \[
    Y \subseteq K \cdot a.
    \]
    Hence, we have that 
    \[
    Y = L \cdot a,
    \]
    where 
    \[
    L := \{u \in K : u \cdot a \in Y\}.
    \]

    The set $L$ is closed in $K$, because $Y$ is closed and the map $u \mapsto u \cdot a$ is continuous. 
    We show that $L$ is a subgroup of $K$.

    Let $u_1, u_2 \in L$. 
    Then there exist $h_1, h_2 \in H$ such that 
    \[
    u_1 \cdot a = h_1 \cdot a, \qquad u_2 \cdot a = h_2 \cdot a.
    \]
    Since $G_k$ is central, the action of $K$ commutes with all translations by elements of $G$.
    Thus $u_2^{-1} \cdot a = h_2^{-1} \cdot a$, and so 
    \[
    u_1 u_2^{-1} \cdot a = u_1 \cdot (h_2^{-1} \cdot a) = h_2^{-1} \cdot (u_1 \cdot a) = h_2^{-1} h_1 \cdot a \in Y.
    \]
    Therefore $u_1 u_2^{-1} \in L$, so $L$ is a subgroup of $K$.
    Since $Y$ is non-trivial, and $Y = L \cdot a$, $L$ is non-trivial. 

    We claim that $m_Y$ is the pushforward of the Haar measure $m_L$ under the map 
    \[
    L \to Y, \qquad u \mapsto u \cdot a.
    \]
    Indeed, if $h \in H$, then there exists a unique $u_0 \in L$ such that 
    \[
    h \cdot a = u_0 \cdot a.
    \]
    But then using the fact that the $K$-action commutes with all translations, 
    \[
    h \cdot (u \cdot a) = u \cdot (h \cdot a) = u_0 u \cdot a, 
    \]
    for all $u \in L$.
    Hence, the invariance of $m_L$ implies that its pushforward is invariant under all translations by elements of $H$. 
    By uniqueness of Haar measure, the pushforward must be $m_Y$.
    
    Therefore, by our choice of $c''$,
    \[
    \int_{Y \times Y} d_X(x,y) \,d m_Y(x) \,d m_Y(y)
    = \int_{L \times L} d_X(u \cdot a, v \cdot a) \,d m_L(u) \,d m_L(v) 
    \geq c'' \geq c.
    \]
    This completes the induction.
\end{proof}

\section{Local rigidity of joinings of nilsystems}\label{sec:joining}

Our proof of Theorem~\ref{thm-main} relies heavily on the construction of joinings of nilsystems.
In the ergodic setting, the systems defined by such joinings are again nilsystems. A proof of the following result can be found in \cite[Chapter~11, Proposition~15]{hk-book}. 

\begin{proposition}\label{prop:joining_nilsystems_is_nilsystem}
    Let $(X = G/\Gamma, \mu, T)$ and $(X' = G'/\Gamma', \mu', T')$ be $k$-step nilsystems, and let $\lambda$ be an ergodic joining of these systems.
    Then there exists a $k$-step subnilsystem $(Y,T\times T')$ of $(X\times X',T\times T')$ such that $\lambda$ is the Haar measure on $Y$.
    In particular, the system defined by the joining is itself a $k$-step nilsystem.
\end{proposition}

Concretely, Proposition~\ref{prop:joining_nilsystems_is_nilsystem} says that there exist a closed subgroup $H \leq G\times G'$, containing the element $(\tau,\tau')$ defining the nilrotation $T\times T'$, and a point $x\in X\times X'$ such that $Y := H\cdot x$ is a subnilmanifold of $X \times X'$ and $\lambda$ is its Haar measure. 

Before proving Proposition \ref{prop:neighbourhood_diag_joining_graph}, let us compare it with the classical case of rotations on compact abelian groups. 
Let  $(Z, m_Z, T)$ be an ergodic compact abelian group rotation\footnote{By a compact abelian group rotation, we mean that $Z$ is a compact metrizable abelian group, $m_Z$ is its Haar probability measure, and $T$ is the rotation by a fixed element of $Z$.}, where
\[T \colon Z \to Z, \qquad T z = z + \alpha\]
is the rotation by $\alpha \in Z$. 
A standard result on joinings of rotations shows that every ergodic self-joining of $(Z, m_Z, T)$ is the image under some translation of the Haar measure of the closed subgroup
\[
H := \overline{\{(n\alpha, n\alpha) : n \in \Z\}} \subseteq Z \times Z;
\]
see for example \cite[Chapter 4, Proposition 7]{hk-book}.
Since $T$ is ergodic,
\[
\overline{\{n \alpha : n \in \Z\}} = Z
\]
and therefore 
\[
H = \Delta := \{(z, z) : z \in Z\}.
\]
It follows that any ergodic self-joining of $Z$ is supported on a translate of the diagonal, hence is the graph joining of a translation 
\[z \mapsto z + s\]
for some $s \in Z$.
Thus in the compact abelian rotation setting, one has a global rigidity result: every ergodic self-joining is the graph joining of a translation. 

This global rigidity fails for more general nilsystems. 
A counterexample already arises in the $2$-step nilpotent Heisenberg nilsystem. 
\begin{example}
Let $G = \R^3$ with multiplication law
\[
(x, y, z) \cdot (x', y', z') := (x + x', y + y', z + z' + x y').
\]
Let $\Gamma = \Z^3 \subseteq G$, and write $X = G/\Gamma$.  
Fix $\tau := (\alpha, \beta, \gamma) \in G$ such that $1, \alpha, \beta$ are linearly independent over $\Q$, and let $T$ be the nilrotation
\[
T(g\Gamma) := \tau g \Gamma.
\]
Then $(X, \mu, T)$ is the ergodic Heisenberg nilsystem.  

Let 
\[
C := \{(0, 0, t) : t \in \R\} \subseteq G
\]
be the center of $G$. 
Choose $s \in \R$ such that $1, \alpha, \beta, \alpha s$ are linearly independent over $\Q$, and set 
\[
a := (0, s, 0). 
\]
Define 
\[
H := \{(g, a g a^{-1} c): g \in G, c \in C\} \leq G \times G.
\]
Since $C$ is closed and central, $H$ is a closed subgroup of $G \times G$. 
Now consider 
\[
Y := H \cdot (e_G \Gamma, a \Gamma) = \{(g\Gamma, agc\Gamma) : g \in G, c \in C\} \subseteq X \times X.
\]
If $\T = \R/\Z$, the map 
\[
\phi \colon X \times \T \to Y, \qquad \phi(g\Gamma, t) := (g\Gamma, ag(0,0,t) \Gamma)
\]
is a well-defined homeomorphism.
Hence $Y$ is compact and therefore a subnilmanifold of $X \times X$.
Let $\lambda$ be the Haar probability measure on $Y$.

Note that 
\[
\tau a = (\alpha, \beta + s, \gamma + \alpha s) = a \tau (0, 0, \alpha s),
\]
and therefore
\[(T \times T) \circ \phi = \phi \circ (T \times R_{\alpha s}),\]
where 
\[
R_{\alpha s} (t) := t + \alpha s \mod 1.
\]
Now $X \times \T$ is itself a nilmanifold, and the choice of $s$ implies that $T \times R_{\alpha s}$ defines an ergodic nilrotation.
Hence $(Y, \lambda, T \times T)$ is ergodic. 

Since $\lambda$ is $T \times T$-invariant, the coordinate projections of $\lambda$ are $T$-invariant probability measures on $X$. 
By unique ergodicity of ergodic nilsystems, both marginals must then agree with $\mu$. 
Hence $\lambda$ is an ergodic self-joining of $X$.

Finally, $\lambda$ is not a graph joining. 
Indeed, for each $g\Gamma \in X$, the fiber of the first coordinate projection $\pi_1 \colon Y \to X$ is homeomorphic to $\T$,
\[
\pi_1^{-1}(g\Gamma) = \{(g\Gamma, agc\Gamma) : c \in C\}.
\]
In particular $\pi_1^{-1}(g\Gamma)$ is non-trivial.
Therefore $\lambda$ cannot be a graph joining. 
\end{example}

We are now ready to prove Proposition \ref{prop:neighbourhood_diag_joining_graph}. 

\begin{proof}[Proof of Proposition \ref{prop:neighbourhood_diag_joining_graph}]
    Fix a compatible metric $d_X$ on $X = G/\Gamma$.
    Define the product metric 
    \[
    d_{X \times X}((x,y), (x',y')) := d_X(x,x') + d_X (y,y').
    \]
    Let $c > 0$ be the constant obtained from Lemma \ref{lem:no_small_subnilmanifolds_averaged} applied to the nilmanifold $X \times X$ with metric $d_{X \times X}$.

    Consider the set 
    \[
    U := \left\{ \lambda \in J_e(X, X) : \int_{X \times X} d_X(x,y) \,d\lambda(x,y) < c/2 \right\}.
    \]
    Since $d_X$ is continuous on $X \times X$, $U$ is open in $J_e(X,X)$ by definition of the weak-* topology.
    Moreover, 
    \[
    \int_{X \times X} d_X(x,y) \,d\mu_\Delta(x,y)
    = \int_{X} d_X(x,x) \,d\mu(x) = 0,
    \]
    so $\mu_\Delta \in U$. 
    Since the metric $d$ on $\mathrm{Pr}(X \times X)$ induces the weak-* topology, there exists $\delta > 0$ such that 
    \[
    \{\lambda \in J_e(X,X) : d(\lambda, \mu_\Delta) < \delta\} \subseteq U.
    \]
    It remains to prove that every $\lambda \in U$ is the graph joining of an automorphism of $(X, \mu, T)$. 

    Let $\lambda \in U$.
    By Proposition \ref{prop:joining_nilsystems_is_nilsystem}, $\lambda$ is the Haar measure of a $T \times T$-invariant subnilmanifold 
    \[
    Y = H \cdot (x_0, y_0) \subseteq X \times X,
    \]
    where $H \leq G \times G$ is a closed subgroup and $(x_0, y_0) \in X \times X$. 

    Let 
    \[
    \pi_1 \colon Y \to X
    \]
    be the first coordinate projection. 
    Since $(\pi_1)_* \lambda = \mu$, and since $\mu$ has full topological support, $\pi_1$ is surjective.
    Indeed $\pi_1(Y)$ is compact, hence closed and has full measure, so $\pi_1(Y) = X$.
    
    Next, we show that $\pi_1$ is injective.

    For $x \in X$, define the fiber
    \[
    Y_x := Y \cap (\{x\} \times X).
    \]
    We first check that each $Y_x$ is a subnilmanifold of $X \times X$.
    Since $\pi_1$ is surjective, $Y_x$ is non-empty, so we can choose some $(x,y) \in Y_x$.
    Then $Y = H \cdot (x,y)$.
    If $x = g\Gamma$, then 
    \[
    Y_x = H_x \cdot (x,y),
    \]
    where 
    \[
    H_x := H \cap ((g \Gamma g^{-1}) \times G).
    \]
    The subgroup $H_x$ is closed in $G \times G$, and $Y_x$ is closed in $X \times X$. 
    Hence $Y_x$ is a subnilmanifold of $X \times X$.

    Let $m_{Y_x}$ denote the Haar measure on $Y_x$.
    We claim that 
    \[
    \lambda = \int_X m_{Y_x} \,d \mu(x)
    \] 
    is the disintegration of $\lambda$ over $\pi_1$.

    Observe that if $h = (h_1, h_2) \in H$, then the translation by $h$ maps $Y_x$ homeomorphically onto $Y_{h_1} \cdot x$, and it maps $m_{Y_x}$ to $m_{Y_{h_1 \cdot x}}$.

    The measurability of $x \mapsto m_{Y_x}$ follows from a measurable section argument. 
    Choose $b_0 \in X$ with $(e_X, b_0) \in Y$. 
    Since $\pi_1$ is surjective, the map 
    \[
    H \to X, \qquad h = (h_1, h_2) \mapsto h_1 \cdot e_X
    \]
    is surjective, hence it admits a measurable section $\sigma \colon X \to H$; see for example \cite[Lemma~7]{baggett1980functional}.
    Then $Y_x = \sigma(x) \cdot Y_{e_X}$, and  
    \[
    m_{Y_x} = (\sigma(x))_* m_{Y_{e_X}},
    \]
    so $x \mapsto \int_{Y_x} f \,d m_{Y_x}$ is measurable for every $f \in C(Y)$.

    Define a probability measure $\widetilde{\lambda}$ on $Y$ by setting
    \[
    \int_Y f \,d\widetilde{\lambda} 
    := \int_X \int_{Y_x} f \,d m_{Y_x} \,d\mu(x)
    \]
    for all $f \in C(Y)$.
    Then by combining the above observation with invariance of $\mu$, we get for any $h = (h_1, h_2) \in H$ that  
    \begin{align*}
        \int_Y f(h \cdot z) \,d\widetilde{\lambda}(z)
        &= \int_X \int_{Y_x} f(h \cdot z) \,d m_{Y_x}(z) \,d\mu(x) \\
        &= \int_X \int_{Y_{h_1 \cdot x}} f(z) \,d m_{Y_{h_1 \cdot x}}(z) \,d\mu(x) \\
        &= \int_X \int_{Y_x} f(z) \,d m_{Y_x}(z) \,d\mu(x) 
        = \int_Y f(z) \,d\widetilde{\lambda}(z),
    \end{align*}
    for all $f \in C(Y)$.
    By uniqueness of the Haar measure on $Y$, we thus have $\widetilde{\lambda} = \lambda$, so the disintegration claim holds.
    
    Now suppose that $\pi_1$ is not injective.
    Then $Y_x$ is non-trivial for some $x \in X$.
    Since $H$ acts transitively on $Y$, and the translation by $h = (h_1, h_2) \in H$ maps $Y_x$ bijectively to $Y_{h_1 \cdot x}$, every fiber must then be non-trivial.  

    Thus by the choice of $c$, 
    \[
    \int_{Y_x \times Y_x} d_{X \times X}(z, z') \,d m_{Y_x}(z) \,d m_{Y_x}(z') \geq c
    \]
    for every $x \in X$.

    But for $z = (x,y)$ and $z' = (x,y')$ in $Y_x$, we have that 
    \[
    d_{X \times X}(z, z') = d_X(y,y') \leq d_X(x,y) + d_X(x,y').
    \]
    Therefore, integrating over $x$, and using the disintegration of $\lambda$ we just proved, we get  
    \begin{align*}
        c &\leq 
        \int_X \int_{Y_x \times Y_x} (d_X(x,y) + d_X(x,y')) \,d m_{Y_x}(x,y) \,d m_{Y_x}(x,y') \,d\mu(x) \\
        &\leq 2 \int_X \int_{Y_x} d_X(x,y) \,d m_{Y_x}(x,y) \,d\mu(x) \\
        &= 2 \int_Y d_X(x,y) \,d\lambda(x,y).
    \end{align*}
    But $\lambda \in U$, so the last expression is strictly bounded by $c$. 
    This is a contradiction, so $\pi_1$ must be injective. 

    We have established that $\pi_1 \colon Y \to X$ is a continuous bijection.
    Since $Y$ is compact, $\pi_1$ is a homeomorphism.

    By symmetry, the same argument applied to the second coordinate shows  that the second coordinate projection $\pi_2 \colon Y \to X$ is also a homeomorphism. 

    Define 
    \[
    S := \pi_2 \circ \pi_1^{-1} \colon X \to X.
    \]
    Then $S$ is a homeomorphism.
    Moreover, since $Y$ is $T \times T$-invariant and $\pi_1$ intertwines $T \times T$ with $T$, we have 
    \[
    S(Tx) = \pi_2(\pi_1^{-1}(Tx)) = \pi_2((T \times T) \pi_1^{-1}(x)) = T(S(x)),
    \]
    for every $x \in X$.
    Thus $S$ commutes with the dynamics. 

    Since $(\pi_1)_* \lambda = \mu$, and $\pi_1$ is a homeomorphism, we have 
    \[
    \lambda = (\pi_1^{-1})_* \mu.
    \]
    Therefore, 
    \[
    S_* \mu = (\pi_2)_* (\pi_1^{-1})_* \mu = (\pi_2)_* \lambda = \mu,
    \]
    so $S$ is an automorphism of $(X, \mu, T)$.

    Finally, we have that 
    \[
    \lambda = (\pi_1^{-1})_* \mu = (\Id_X, \pi_2 \circ \pi_1^{-1})_* \mu = (\Id_X, S)_* \mu, 
    \]
    so $\lambda$ is precisely the graph joining associated to $S$.
\end{proof}

As an immediate consequence, factor maps into an ergodic nilsystem are locally rigid modulo automorphisms. We shall not use this corollary in the remainder of the proof of Theorem~\ref{thm-main}.

\begin{corollary}\label{cor:local_rigidity_factor_maps}
Let $(X,\mu,T)$ be an ergodic measure-preserving dynamical system, let
$(Y,\nu,R)$ be an ergodic nilsystem, and let $\pi,\psi \colon X \to Y$ 
be factor maps. Then there exists $\delta>0$
such that, if $d\bigl((\pi,\psi)_*\mu,\nu_\Delta\bigr)<\delta$, there exists an automorphism $S$ of $(Y,\nu,R)$ such that $\psi = S \circ \pi$ $\mu$-almost surely. 
\end{corollary}

\begin{proof}
Set $\lambda := (\pi,\psi)_*\mu \in J_e(Y, Y)$. By Proposition~\ref{prop:neighbourhood_diag_joining_graph}, there is $\delta>0$ such that, if $d(\lambda,\nu_\Delta)<\delta$, then $\lambda$ is the graph joining of an automorphism $S$ of $(Y,\nu,R)$, that is, $\lambda$ is supported on the graph $\Gamma_S := \{(y,Sy): y \in Y\}$. Therefore,
$1 = \lambda(\Gamma_S)
  = \mu\bigl((\pi,\psi)^{-1}(\Gamma_S)\bigr)
  = \mu\bigl(\{x \in X : \psi(x)=S(\pi(x))\}\bigr)$. 
\end{proof}

\section{A special case}\label{sec:special}
We now begin the proof of the main theorem.
The idea is to first establish the theorem in the special case where the factor map is a compact abelian group extension, and then reduce the general case to this situation.
This section is devoted to that special case.

We will use the following theorem, which is proved as \cite[Chapter~13, Theorem~11]{hk-book}, and which is originally due to Parry \cite{parry}.  

\begin{theorem}\label{thrm:factor_of_nilsystem_is_nilsystem}
    Let $k \geq 1$.
    Then every factor of an ergodic $k$-step nilsystem is again an ergodic $k$-step nilsystem.
\end{theorem}

We now turn to the special case of the main theorem.

\begin{theorem}[Compact abelian group extension case]\label{thrm:special_case}
    Let $(X,\mu,T)$ be an ergodic $k$-step pro-nilsystem, and let
    \[
    \pi \colon (X,\mu,T)\to (Y,\nu,S)
    \]
    be a factor map such that $(X,\mu,T)$ is a skew-product extension of $(Y,\nu,S)$ by a compact abelian group $K$.
    Then $(Y,\nu,S)$ is a $k$-step pro-nilsystem.
\end{theorem}

\begin{proof}
    Write
    \[
    (X,\mu,T)=\varprojlim (X_i,\mu_i,T_i),
    \]
    where each $(X_i,\mu_i,T_i)$ is an ergodic $k$-step nilsystem, and let
    \[
    p_i\colon X\to X_i
    \]
    denote the factor maps.

    By assumption, $X$ is of the form 
    \[
    (X,\mu,T)=(Y\times K,\nu\times m_K,T_\rho),
    \]
    where $m_K$ is Haar measure on $K$, $\rho\colon Y\to K$ is a measurable map, and
    \[
    T_\rho(y,g)=(Sy,\rho(y)+g).
    \]
    For $u\in K$, write
    \[
    V_u\colon X\to X,\qquad V_u(y,g)=(y,u+g),
    \]
    for the corresponding vertical rotation.

    For each $i$ and each $u\in K$, define a measure on $X_i\times X_i$ by
    \[
    \lambda_u^{(i)}:=(p_i,p_i\circ V_u)_*\mu.
    \]
    Since $p_i$ is a factor map, $\lambda_u^{(i)}$ is a self-joining of $(X_i,\mu_i,T_i)$.
    Moreover, the system
    \[
    (X_i\times X_i,\lambda_u^{(i)},T_i\times T_i)
    \]
    is a factor of the ergodic system $(X,\mu,T)$ via the map $(p_i,p_i\circ V_u)$, and is therefore ergodic.
    Thus
    \[
    \lambda_u^{(i)}\in J_e(X_i,X_i).
    \]

    For each $i$, define
    \[
    H_i:=\{u\in K : \lambda_u^{(i)} \text{ is the graph joining of an automorphism of }X_i\}.
    \]
    We claim that $H_i$ is an open subgroup of $K$.

    Let $u,v\in H_i$.
    Then there exist automorphisms $\phi_u,\phi_v\colon X_i\to X_i$ such that
    \[
    p_i\circ V_u=\phi_u\circ p_i \qquad \mu\text{-a.e.}
    \]
    and
    \[
    p_i\circ V_v=\phi_v\circ p_i \qquad \mu\text{-a.e.}
    \]
    Hence
    \[
    p_i\circ V_{u+v}
    =\phi_u\circ \phi_v\circ p_i
    \qquad \mu\text{-a.e.}
    \]
    so $\lambda_{u+v}^{(i)}$ is the graph joining associated to the automorphism $\phi_u\circ\phi_v$.
    Thus $u+v\in H_i$.

    Similarly,
    \[
    \phi_u^{-1}\circ p_i
    =\phi_u^{-1}\circ p_i\circ V_u\circ V_{-u}
    =\phi_u^{-1}\circ \phi_u\circ p_i\circ V_{-u}
    =p_i\circ V_{-u}
    \qquad \mu\text{-a.e.},
    \]
    and therefore $\lambda_{-u}^{(i)}$ is the graph joining associated to $\phi_u^{-1}$.
    Hence $-u\in H_i$, and so $H_i$ is a subgroup of $K$.

    By Proposition~\ref{prop:neighbourhood_diag_joining_graph}, there exists an open neighbourhood
    \[
    U_i\subseteq J_e(X_i,X_i)
    \]
    of the diagonal joining such that every element of $U_i$ is the graph joining of an automorphism of $X_i$.
    
    We claim that the map
\[
K \to J_e(X_i,X_i), \qquad u \mapsto \lambda_u^{(i)}
\]
is continuous.
Let $u_n \to u$ in $K$, and let $f,g \in C(X_i)$.
Then
\[
\int_{X_i\times X_i} f(x)g(x')\, d\lambda^{(i)}_{u_n}
=
\int_X f(p_i(x))\, g(p_i(V_{u_n}x))\, d\mu(x).
\]
Write $F = f\circ p_i$ and $G = g\circ p_i$, viewed as elements of $L^2(X,\mu)$.
Since the vertical rotations $V_u$ define a measure-preserving action of the compact group $K$ on $X$, the associated Koopman representation
\[
U_u h := h\circ V_u
\]
on $L^2(X,\mu)$ is strongly continuous.
Hence
\[
\|U_{u_n}G - U_uG\|_{L^2(\mu)} \to 0,
\]
and therefore
\[
\int_X F \cdot U_{u_n}G \, d\mu \to \int_X F \cdot U_uG \, d\mu.
\]
That is,
\[
\int_{X_i\times X_i} f(x)g(x')\, d\lambda^{(i)}_{u_n}
\to
\int_{X_i\times X_i} f(x)g(x')\, d\lambda^{(i)}_u.
\]
Since finite linear combinations of functions of the form $(x,x')\mapsto f(x)g(x')$ are dense in $C(X_i\times X_i)$, it follows that
\[
\lambda^{(i)}_{u_n} \to \lambda^{(i)}_u
\]
in the weak-* topology.

    Since
    \[
    \lambda_0^{(i)}=(p_i,p_i)_*\mu
    \]
    is the diagonal joining, the preimage of $U_i$ under this map is an open neighbourhood of $0$ contained in $H_i$.
    Therefore $H_i$ is an open subgroup of $K$.

    Since $K$ is compact and $H_i$ is open, the quotient $K/H_i$ is finite.
    Choose a finite set $R_i\subseteq K$ of coset representatives such that
    \[
    R_i+H_i=K.
    \]
    By enlarging the sets $R_i$ if necessary, we may assume that $0\in R_i$ for every $i$, and that
    \[
    R_i\subseteq R_j \qquad \text{whenever } i\leq j.
    \]

    For each $u\in H_i$, let
    \[
    V_u^{(i)}\colon (X_i,\mu_i,T_i)\to (X_i,\mu_i,T_i)
    \]
    be an automorphism whose graph joining is $\lambda_u^{(i)}$.
    Then
    \[
    V_u^{(i)}\circ p_i=p_i\circ V_u \qquad \mu\text{-a.e.}
    \]
    and hence
    \[
    V_u^{-1}(p_i^{-1}(\mathcal X_i)) =_\mu p_i^{-1}(\mathcal X_i),
    \]
    where $\mathcal X_i$ denotes the sigma-algebra of the factor $X_i$.
    Thus the sigma-algebra corresponding to $X_i$ is invariant under vertical rotations by elements of $H_i$.

    We now enlarge $X_i$ to a factor invariant under all vertical rotations.
    For each $i$, set
    \[
    \widetilde X_i:=X_i^{R_i}=\prod_{r\in R_i} X_i,
    \]
    and define
    \[
    q_i\colon X\to \widetilde X_i,\qquad
    q_i(x):=(p_i(V_r x))_{r\in R_i}.
    \]
    Let
    \[
    \widetilde\mu_i:=(q_i)_*\mu,
    \qquad
    \widetilde T_i((x_r)_{r\in R_i}):=(T_i x_r)_{r\in R_i}.
    \]
    Then
    \[
    q_i\colon (X,\mu,T)\to (\widetilde X_i,\widetilde\mu_i,\widetilde T_i)
    \]
    is a factor map, so $(\widetilde X_i,\widetilde\mu_i,\widetilde T_i)$ is ergodic.
    Moreover, it is an ergodic finite self-joining of the nilsystem $X_i$, and hence is itself a $k$-step nilsystem by repeated application of Proposition~\ref{prop:joining_nilsystems_is_nilsystem}.

    The sigma-algebra corresponding to $\widetilde X_i$ is
    \[
    q_i^{-1}(\widetilde{\mathcal X}_i)
    = \bigvee_{r\in R_i} V_r^{-1}(p_i^{-1}(\mathcal X_i)).
    \]
    We claim that this sigma-algebra is invariant under all vertical rotations.
    Let $u\in K$.
    For each $r\in R_i$, since $R_i+H_i=K$, there exist $s(r)\in R_i$ and $h(r)\in H_i$ such that
    \[
    r+u=s(r)+h(r).
    \]
    Therefore
    \[
    V_u^{-1}(V_r^{-1}(p_i^{-1}(\mathcal X_i)))
    =V_{r+u}^{-1}(p_i^{-1}(\mathcal X_i))
    =V_{s(r)}^{-1}\bigl(V_{h(r)}^{-1}(p_i^{-1}(\mathcal X_i))\bigr)
    =_\mu V_{s(r)}^{-1}(p_i^{-1}(\mathcal X_i)),
    \]
    using the $H_i$-invariance of $p_i^{-1}(\mathcal X_i)$.
    Hence
    \[
    V_u^{-1}(q_i^{-1}(\widetilde{\mathcal X}_i))
    =_\mu \bigvee_{r\in R_i} V_{s(r)}^{-1}(p_i^{-1}(\mathcal X_i))
    \subseteq_\mu q_i^{-1}(\widetilde{\mathcal X}_i).
    \]
    Since $V_u$ is invertible, the reverse inclusion also holds, and thus
    \[
    V_u^{-1}(q_i^{-1}(\widetilde{\mathcal X}_i)) =_\mu q_i^{-1}(\widetilde{\mathcal X}_i)
    \]
    for every $u\in K$.

    Since $0\in R_i$, each factor $\widetilde X_i$ lies above $X_i$.
    Moreover, because the families $R_i$ and $X_i$ are increasing, the sigma-algebras $q_i^{-1}(\widetilde{\mathcal X}_i)$ are increasing.

    For each $i$, define
    \[
    \mathcal W_i
    := \overline{\pi^{-1}(\mathcal Y)}^{\,\mu}\cap
       \overline{q_i^{-1}(\widetilde{\mathcal X}_i)}^{\,\mu},
    \]
    where the bar denotes $\mu$-completion.
    Let $W_i$ be the corresponding factor of $X$.
    By construction, $W_i$ is a factor of $\widetilde X_i$, and hence, by Theorem~\ref{thrm:factor_of_nilsystem_is_nilsystem}, each $W_i$ is an ergodic $k$-step nilsystem.

    We now show that the sigma-algebra $\pi^{-1}(\mathcal Y)$ is generated by the increasing family $(\mathcal W_i)$.
    Let $f\in L^2(\pi^{-1}(\mathcal Y))$ and $\eps>0$.     Since $X=\varprojlim X_i$ and $\widetilde X_i$ lies above $X_i$, there exist $i_0$ and $ g\in L^2(q_{i_0}^{-1}(\widetilde{\mathcal X}_{i_0}))$ 
    such that $ \|g-f\|_{L^2(\mu)}<\eps$. 
    Set
    \[
    h:=\E_\mu(g\mid \pi^{-1}(\mathcal Y)).
    \]
    Since $X$ is an extension of $Y$ by $K$, we have
    \[
    h(x)=\int_K g(V_u x)\,dm_K(u)
    \qquad \mu\text{-a.e.}
    \]
    Because $q_{i_0}^{-1}(\widetilde{\mathcal X}_{i_0})$ is invariant under all vertical rotations, each function $g\circ V_u$ is measurable with respect to this sigma-algebra.
    Hence $h$ is also measurable with respect to $q_{i_0}^{-1}(\widetilde{\mathcal X}_{i_0})$.
    By construction, $h$ is $\pi^{-1}(\mathcal Y)$-measurable as well.
    Therefore $h\in L^2(\mathcal W_{i_0})$. 

    Moreover,
    \[
    \|h-f\|_{L^2(\mu)}
    =\|\E_\mu(g-f\mid \pi^{-1}(\mathcal Y))\|_{L^2(\mu)}
    \leq \|g-f\|_{L^2(\mu)}
    <\eps.
    \]
    Since $f$ and $\eps$ were arbitrary, it follows that
    \[
    \pi^{-1}(\mathcal Y)\subseteq_\mu \bigvee_i \mathcal W_i.
    \]
    The reverse inclusion is immediate from the definition of $\mathcal W_i$.
    Hence
    \[
    \pi^{-1}(\mathcal Y)=_\mu \bigvee_i \mathcal W_i.
    \]

    Thus the factor $(Y,\nu,S)$ is the inverse limit, in the measure-theoretic sense, of the increasing family of $k$-step nilsystems $W_i$.
    Therefore $(Y,\nu,S)$ is a $k$-step pro-nilsystem.
\end{proof}

\section{Proof of Theorem \ref{thm-main}}\label{Sec:proof} 

We are now ready to prove the main theorem.
The idea is to construct, for a pro-nilsystem, a tower of factors by taking the corresponding tower at each finite stage of the inverse limit and then passing this construction to the inverse limit.
The downward induction argument used in the proof may be viewed as a generalization to the pro-nilsystem setting of the argument used by Host and Kra in \cite[Chapter~13, Theorem~11]{hk-book} to show that factors of ergodic nilsystems are again nilsystems. 
We require one further result, after which we will prove the main theorem.
For the explicit construction of the tower of factors of a nilsystem, see \cite[Chapter~11, Section~1]{hk-book}.

\begin{proposition}\label{prop:inverse_limit_extension}
    Let $k \geq 2$, and suppose that
    \[
    (X,\mu,T)=\varprojlim (X_i,\mu_i,T_i),
    \]
    where each $(X_i,\mu_i,T_i)$ is an ergodic $k$-step nilsystem.
    For each $i$, let $(Z_i,\nu_i,S_i)$ denote the $(k-1)$-step factor in the tower of $X_i$, and let $K_i$ be the top structure group, so that $X_i$ is an extension of $Z_i$ by $K_i$.

    Then the systems $(Z_i,\nu_i,S_i)$ and the groups $K_i$ admit inverse system structures such that the inverse limit $(X,\mu,T)$ is an extension of
    \[
    (Z,\nu,S):=\varprojlim (Z_i,\nu_i,S_i)
    \]
    by the compact abelian group
    \[
    K:=\varprojlim K_i.
    \]
\end{proposition}

\begin{proof}
    For each $i$, write
    \[
    X_i=G_i/\Gamma_i
    \]
    in reduced form, meaning that $G_i$ is generated by the connected component of the identity together with the element defining the nilrotation, and that $\Gamma_i$ contains no non-trivial normal subgroup of $G_i$; see \cite[Chapter~11, Section~1]{hk-book}.
    Let $\tau_i\in G_i$ denote the element defining the nilrotation $T_i$.

    Write
    \[
    \pi_{i,j}\colon X_j\to X_i
    \]
    for the factor maps defining the inverse system.
    Since the systems $X_i$ are ergodic nilsystems, \cite[Chapter~13, Theorem~5]{hk-book} implies that for each $i\leq j$, after modifying $\pi_{i,j}$ on a null set, we may assume that $\pi_{i,j}$ is a continuous factor map of the form
    \[
    \pi_{i,j}(g\Gamma_j)=a_{i,j}\Psi_{i,j}(g)\Gamma_i,
    \]
    where $\Psi_{i,j}\colon G_j\to G_i$ is a continuous surjective homomorphism, $a_{i,j}\in G_i$, and
    \[
    \Psi_{i,j}(\Gamma_j)\subseteq \Gamma_i,
    \qquad
    \Psi_{i,j}(\tau_j)=a_{i,j}^{-1}\tau_i a_{i,j}.
    \]
    Since Haar measure has full support, continuous maps agreeing almost everywhere agree everywhere.
    Thus, after modifying on null sets once and for all, we may assume that
    \[
    \pi_{i,l}=\pi_{i,j}\circ \pi_{j,l}
    \]
    holds everywhere.

    We may therefore form the topological inverse limit
    \[
    X^{\mathrm{top}}:=\varprojlim X_i
    =\left\{(x_i)_i\in \prod_i X_i \,\middle|\, \pi_{i,j}(x_j)=x_i \text{ for all } i\leq j\right\},
    \]
    equipped with the transformation
    \[
    T(x_i)_i=(T_i x_i)_i.
    \]
    Let $\pi_i\colon X\to X_i$ denote the coordinate projections.

    Since each $(X_i,T_i)$ is an ergodic nilsystem, it is uniquely ergodic.
    Unique ergodicity passes to the inverse limit, so $X^{\mathrm{top}}$ is uniquely ergodic as well.
    Let $\mu$ denote its unique $T$-invariant measure.
    Then for every $i$,
    \[
    (T_i)_*(\pi_i)_*\mu = (\pi_i)_*T_*\mu = (\pi_i)_*\mu,
    \]
    and by unique ergodicity of $X_i$ it follows that
    \[
    (\pi_i)_*\mu=\mu_i.
    \]
    We have thus constructed a topological inverse limit $X^{\mathrm{top}}$ with its unique invariant measure $\mu$, and for each $i$ the coordinate projection $\pi_i$ satisfies $(\pi_i)_*\mu=\mu_i$. Hence $(X^{\mathrm{top}},\mu,T)$ realizes the same inverse limit as the original system $(X,\mu,T)$ in the measure-theoretic sense. Replacing $(X,\mu,T)$ by this isomorphic model, we may therefore assume from now on that $X$ is this topological inverse limit.

    For each $i$ and each $j$, let $G_{i,j}$ denote the $j$th term of the lower central series of $G_i$.
    Then the $(k-1)$-step factor of $X_i$ is, by definition,
    \[
    Z_i=G_i/(G_{i,k}\Gamma_i),
    \]
    equipped with Haar measure $\nu_i$ and the nilrotation $S_i$ induced by $\tau_i$.
    Let
    \[
    q_i\colon X_i\to Z_i
    \]
    denote the natural projection.

    For $i\leq j$, define
    \[
    \phi_{i,j}\colon Z_j\to Z_i,\qquad
    \phi_{i,j}\bigl(g(G_{j,k}\Gamma_j)\bigr)
    = a_{i,j}\Psi_{i,j}(g)(G_{i,k}\Gamma_i).
    \]
    This is well defined because
    \[
    \Psi_{i,j}(G_{j,k})=G_{i,k}
    \]
    and $\Psi_{i,j}(\Gamma_j)\subseteq \Gamma_i$.
    Moreover,
    \[
    q_i\circ \pi_{i,j}=\phi_{i,j}\circ q_j,
    \]
    so $\phi_{i,j}$ is continuous and surjective.
    It also intertwines the nilrotations:
    \[
    \phi_{i,j}\circ S_j = S_i\circ \phi_{i,j},
    \]
    since $\Psi_{i,j}(\tau_j)=a_{i,j}^{-1}\tau_i a_{i,j}$.

    If $i\leq j\leq l$, then
    \[
    \phi_{i,l}\circ q_l
    = q_i\circ \pi_{i,l}
    = q_i\circ \pi_{i,j}\circ \pi_{j,l}
    = \phi_{i,j}\circ q_j\circ \pi_{j,l}
    = \phi_{i,j}\circ \phi_{j,l}\circ q_l.
    \]
    Since $q_l$ is surjective, it follows that
    \[
    \phi_{i,l}=\phi_{i,j}\circ \phi_{j,l}.
    \]
    Thus $\{(Z_i,S_i),\phi_{i,j}\}$ is an inverse system.
    Let
    \[
    Z=\varprojlim Z_i,
    \qquad
    S(z_i)_i=(S_i z_i)_i.
    \]
    Exactly as above, $Z$ is uniquely ergodic.
    Let $\nu$ denote its unique invariant measure.
    Then
    \[
    (Z,\nu,S)=\varprojlim (Z_i,\nu_i,S_i)
    \]
    in the measure-theoretic sense.

    We now turn to the top structure groups.
    By definition,
    \[
    K_i=G_{i,k}/(G_{i,k}\cap \Gamma_i).
    \]
    Since $X_i=G_i/\Gamma_i$ is in reduced form, $\Gamma_i$ contains no non-trivial normal subgroup of $G_i$.
    But $G_{i,k}\cap \Gamma_i$ is central, hence normal, so it must be trivial.
    Thus there is a canonical identification
    \[
    K_i\cong G_{i,k}.
    \]
    Under this identification, $K_i$ acts freely on $X_i$ by right translation,
    \[
    V_u(g\Gamma_i)=gu\Gamma_i,
    \qquad u\in K_i,
    \]
    and the quotient is exactly $Z_i$.

    For $i\leq j$, define
    \[
    \alpha_{i,j}\colon K_j\to K_i
    \]
    to be the restriction of $\Psi_{i,j}$ to $G_{j,k}$, viewed under the above identification.
    This is a continuous surjective homomorphism, and $\pi_{i,j}$ is equivariant:
    \[
    \pi_{i,j}(V_u x)=V_{\alpha_{i,j}(u)}(\pi_{i,j}x)
    \qquad x\in X_j,\ u\in K_j.
    \]
    If $i\leq j\leq l$, then for $x\in X_l$ and $u\in K_l$,
    \[
    V_{\alpha_{i,l}(u)}(\pi_{i,l}x)
    = \pi_{i,l}(V_u x)
    = \pi_{i,j}(\pi_{j,l}(V_u x))
    = V_{\alpha_{i,j}(\alpha_{j,l}(u))}(\pi_{i,l}x).
    \]
    Since the $K_i$-action on $X_i$ is free,
    \[
    \alpha_{i,l}=\alpha_{i,j}\circ \alpha_{j,l}.
    \]
    Thus $\{K_i,\alpha_{i,j}\}$ is an inverse system of compact abelian groups.
    Let
    \[
    K=\varprojlim K_i.
    \]

    We now define an action of $K$ on $X$ coordinatewise:
    \[
    V_u(x)=\bigl(V_{u_i}(x_i)\bigr)_i,
    \qquad
    u=(u_i)_i\in K,\ x=(x_i)_i\in X.
    \]
    This is well defined by the equivariance relation above.
    It is a continuous action by homeomorphisms, it commutes with $T$, and it is free because each action of $K_i$ on $X_i$ is free.

    Define
    \[
    q\colon X\to Z,\qquad q((x_i)_i)=(q_i(x_i))_i.
    \]
    This is well defined because
    \[
    \phi_{i,j}(q_j(x_j)) = q_i(\pi_{i,j}(x_j))=q_i(x_i),
    \qquad x = (x_i)_i\in X.
    \]
    Moreover, $q$ is continuous and satisfies
    \[
    q\circ T = S\circ q.
    \]

    We claim that the fibers of $q$ are exactly the $K$-orbits.
    First, if $y=V_u(x)$ for some $u\in K$, then for each $i$,
    \[
    q_i(y_i)=q_i(V_{u_i}(x_i))=q_i(x_i),
    \]
    since $q_i$ is the quotient by the $K_i$-action.
    Hence $q(y)=q(x)$.

    Conversely, suppose $q(x)=q(y)$, and write $x=(x_i)_i$, $y=(y_i)_i$.
    Then for each $i$,
    \[
    q_i(x_i)=q_i(y_i),
    \]
    so, since the fibers of $q_i$ are precisely the $K_i$-orbits, there exists $u_i\in K_i$ such that
    \[
    V_{u_i}(x_i)=y_i.
    \]
    For $i\leq j$, we have
    \[
    V_{u_i}(x_i)
    = y_i
    = \pi_{i,j}(y_j)
    = \pi_{i,j}(V_{u_j}(x_j))
    = V_{\alpha_{i,j}(u_j)}(x_i).
    \]
    By freeness of the $K_i$-action,
    \[
    u_i=\alpha_{i,j}(u_j).
    \]
    Thus $u=(u_i)_i\in K$, and by construction $V_u(x)=y$.

    Hence the fibers of $q$ are exactly the $K$-orbits.
    Thus
    \[
    q\colon (X,T)\to (Z,S)
    \]
    is a topological extension by the compact abelian group $K$.

    Finally, for any $u\in K$, the homeomorphism $V_u$ commutes with $T$.
    Therefore $(V_u)_*\mu$ is $T$-invariant.
    By unique ergodicity of $X$,
    \[
    (V_u)_*\mu=\mu.
    \]
    Thus $\mu$ is $K$-invariant.
    It follows that
    \[
    q\colon (X,\mu,T)\to (Z,\nu,S)
    \]
    is an extension by $K$ in the measure-theoretic sense; see \cite[Chapter~5, Proposition~1]{hk-book}.
\end{proof}

We are now in a position to prove the main theorem.

\begin{proof}[Proof of Theorem~\ref{thm-main}]
The case $k=1$ follows from the fact that every ergodic compact abelian group rotation is the inverse limit of compact abelian Lie group rotations by Pontryagin duality. 

    Let $k\geq 2$, and let 
    \[
    (X,\mu,T)=\varprojlim (X^{(i)},\mu^{(i)},T^{(i)})
    \]
    be an inverse limit of ergodic $k$-step nilsystems, and let
    \[
    \pi\colon (X,\mu,T)\to (Y,\nu,S)
    \]
    be a factor map.
    We will show that $(Y,\nu,S)$ is itself an inverse limit of $k$-step nilsystems.

    For each $i$, let
    \[
    X^{(i)} = Z_k^{(i)} \to Z_{k-1}^{(i)} \to \cdots \to Z_1^{(i)} \to Z_0^{(i)}=\{*\}
    \]
    be the tower of factors of $X^{(i)}$.
    For $2\leq j\leq k$, let $K_j^{(i)}$ denote the structure group of the extension
    \[
    Z_j^{(i)}\to Z_{j-1}^{(i)}.
    \]

    Applying Proposition~\ref{prop:inverse_limit_extension} to the inverse system $\{X^{(i)}\}$, we obtain an induced inverse system structure on $\{Z_{k-1}^{(i)}\}$ and $\{K_k^{(i)}\}$ such that
    \[
    Z_k:=X=\varprojlim X^{(i)}
    \]
    is an extension of
    \[
    Z_{k-1}:=\varprojlim Z_{k-1}^{(i)}
    \]
    by the compact abelian group
    \[
    K_k:=\varprojlim K_k^{(i)}.
    \]

    We now continue this construction downward.
    For each $i$ and each $2\leq j\leq k$, the $(j-1)$-step factor in the tower of factors of $Z_j^{(i)}$ is precisely $Z_{j-1}^{(i)}$.
    Suppose $2\leq j\leq k-1$, and assume that we have already constructed an inverse system structure on the ergodic $j$-step nilsystems $\{Z_j^{(i)}\}$, with inverse limit
    \[
    Z_j=\varprojlim Z_j^{(i)}.
    \]
    Applying Proposition~\ref{prop:inverse_limit_extension} again, we obtain an induced inverse system structure on $\{Z_{j-1}^{(i)}\}$ and $\{K_j^{(i)}\}$ such that $Z_j$ is an extension of
    \[
    Z_{j-1}:=\varprojlim Z_{j-1}^{(i)}
    \]
    by the compact abelian group
    \[
    K_j:=\varprojlim K_j^{(i)}.
    \]

    Repeating this for each $j\geq 2$, and adjoining the trivial system at the bottom, we obtain a tower
    \[
    X=Z_k\to Z_{k-1}\to \cdots \to Z_1\to Z_0=\{*\},
    \]
    where each $Z_j=\varprojlim Z_j^{(i)}$ is an inverse limit of $j$-step nilsystems, and for each $j\geq 2$ the factor map
    \[
    Z_j\to Z_{j-1}
    \]
    is an extension by the compact abelian group $K_j$.

    Note also that $Z_1$ is a rotation on a compact abelian group, since it is an inverse limit of $1$-step nilsystems.
    Thus, if we set
    \[
    K_1:=Z_1,
    \]
    then the map $Z_1\to Z_0$ is also an extension by the compact abelian group $K_1$.

    From now on, write
    \[
    p_j\colon X\to Z_j
    \]
    for the factor maps in this tower.

    For $0\leq j\leq k$, define
    \[
    \mathcal W_j:=\pi^{-1}(\mathcal Y)\vee p_j^{-1}(\mathcal Z_j),
    \]
    where script letters denote the corresponding sigma-algebras.
    Each $\mathcal W_j$ is an invariant sub-sigma-algebra of $\mathcal X$, and we write $(W_j,\lambda_j,R_j)$ 
    for the associated factor.

    We claim, by downward induction on $j$, that each $(W_j,\lambda_j,R_j)$ is an inverse limit of $k$-step nilsystems.

    For $j=k$, we have $\mathcal W_k=\mathcal X$, since $Z_k=X$.
    Hence $(W_k,\lambda_k,R_k)$ is isomorphic to $(X,\mu,T)$, and is therefore an inverse limit of $k$-step nilsystems by assumption.

    Now let $j\leq k-1$, and assume that $(W_{j+1},\lambda_{j+1},R_{j+1})$ is an inverse limit of $k$-step nilsystems.
    By definition,
    \[
    \mathcal W_{j+1}
    = \pi^{-1}(\mathcal Y)\vee p_{j+1}^{-1}(\mathcal Z_{j+1})
    = \mathcal W_j \vee p_{j+1}^{-1}(\mathcal Z_{j+1}).
    \]
    Since $Z_{j+1}\to Z_j$ is an extension by the compact abelian group $K_{j+1}$, \cite[Chapter~5, Lemma~18]{hk-book} implies that there exists a closed subgroup $H\leq K_{j+1}$ 
    such that $W_{j+1}$ is an extension of $W_j$ by the compact abelian group $H$.

    By the induction hypothesis, $W_{j+1}$ is an inverse limit of $k$-step nilsystems.
    It is also ergodic, since it is a factor of the ergodic system $X$.
    Therefore Theorem~\ref{thrm:special_case} applies, and we conclude that $(W_j,\lambda_j,R_j)$ is an inverse limit of $k$-step nilsystems.

    By downward induction, $(W_0,\lambda_0,R_0)$ is an inverse limit of $k$-step nilsystems.
    But $\mathcal W_0=\pi^{-1}(\mathcal Y)$,     so $(W_0,\lambda_0,R_0)$ is isomorphic to $(Y,\nu,S)$.
    This completes the proof.
\end{proof}

\section{Proof of Theorem \ref{thrm:topological_version_main_thrm}}\label{sec:topological}

Throughout this section, a topological $k$-step pro-nilsystem is a topological dynamical system, which is isomorphic in the topological dynamical sense to an inverse limit of $k$-step nilsystems. 

As is explained in \cite[Chapter 13, Section 3]{hk-book}, if $(X, \mu, T)$ is an ergodic $k$-step pro-nilsystem, then after replacing it by an isomorphic system, we may assume that $(X, T)$ is a uniquely ergodic topological $k$-step pro-nilsystem, with $\mu$ as its unique invariant probability measure. 
For the rest of this section, every ergodic $k$-step pro-nilsystem is understood to be represented in this way: we take the corresponding uniquely ergodic topological $k$-step pro-nilsystem model, and equip it with its unique invariant probability measure.

We will need the following elementary criterion for upgrading certain measure-theoretic isomorphisms to topological ones. 

\begin{lemma}\label{lem:upgrade_mble_to_top_isom}
    Let $(X, T)$ be a distal uniquely ergodic topological dynamical system, and let $\mu$ be its unique invariant probability measure. 
    Let 
    \[
    \pi \colon (X, T) \to (Y, S)
    \]
    be a topological factor map, and set $\nu := \pi_* \mu$.
    If $\pi$ is a measure-theoretic isomorphism from $(X, \mu, T)$ to $(Y, \nu, S)$, then $\pi$ is a topological isomorphism.
\end{lemma}

Factor maps of this type were studied by Downarowicz and Glasner in \cite{downarowicz2016isomorphic}.
In particular, this lemma can be derived quite easily from their Proposition 2.5.

We remark at this point that topological pro-nilsystems are distal. 
Indeed, nilsystems are distal, see \cite[Chapter 11, Theorem 2]{hk-book}, and distality is preserved under inverse limits.

The next proposition is the main technical point. 

\begin{proposition}\label{prop:factors_between_pronilsystems}
    Let $(X, \mu, T)$ and $(Y, \nu, S)$ be ergodic $k$-step pro-nilsystems, equipped with their topological structures as above. 
    If 
    \[
    \pi \colon (X, \mu, T) \to (Y, \nu, S)
    \]
    is a measure-theoretic factor map, then $\pi$ agrees almost everywhere with a topological factor map 
    \[
    \pi^{\mathrm{top}} \colon (X, T) \to (Y, S). 
    \]
    Moreover, if $\pi$ is a measure-theoretic isomorphism, then $\pi^{\mathrm{top}}$ is a topological isomorphism.
\end{proposition}

This result is proved in \cite[Theorem A.1]{host2010nilsequences} and \cite[Chapter 13, Proposition 15]{hk-book}, using the theory of dynamical dual functions.
We sketch a different proof, which does not rely on Host--Kra theory. 

\begin{proof}[Sketch of proof]
    Write 
    \[
    (X, T) = \varprojlim (X_i, T_i), \qquad (Y, S) = \varprojlim (Y_i, S_i)
    \]
    as topological inverse limits of uniquely ergodic $k$-step nilsystems, and let $p_i \colon X \to X_i$ and $q_i \colon Y \to Y_i$ be the corresponding factor maps. 

    Let 
    \[
    \lambda := (\Id_X, \pi)_* \mu
    \]
    be the graph joining of $\pi$.
    For each $i$, set 
    \[
    \lambda_i := (p_i \times q_i)_* \lambda.
    \]
    Then $\lambda_i$ is an ergodic joining of the nilsystems $X_i$ and $Y_i$ equipped with their Haar measures. 
    Hence by Proposition 3.1, $\lambda_i$ is the Haar measure of an invariant subnilmanifold 
    \[
    Z_i \subseteq X_i \times Y_i, 
    \]
    and hence $(Z_i, T_i \times S_i)$ is a uniquely ergodic $k$-step nilsystem. 

    There is an induced inverse system structure on the systems $(Z_i, T_i \times S_i)$, and the inverse limit can be identified as 
    \[
    (Z, T \times S) \cong \varprojlim (Z_i, T_i \times S_i)
    \]
    where $Z \subseteq X \times Y$ is some $T \times S$-invariant closed subset.
    Then $(Z, T \times S)$ is uniquely ergodic, and $\lambda$ is its unique invariant probability measure.

    Let 
    \[
    r_X \colon Z \to X, \qquad r_Y \colon Z \to Y
    \]
    be the coordinate projections, which are topological factor maps. 
    Since $\lambda$ is the graph joining of $\pi$, the map $r_X$ is a measure-theoretic isomorphism. 
    Since $(Z, T \times S)$ is distal and uniquely ergodic, Lemma \ref{lem:upgrade_mble_to_top_isom} implies that $r_X$ is a topological isomorphism. 
    Therefore 
    \[
    \pi^{\mathrm{top}} := r_Y \circ r_X^{-1}
    \]
    is a topological factor map, which by construction agrees almost everywhere with $\pi$.
    If $\pi$ is a measure-theoretic isomorphism, then the same argument, applied to $r_Y$, shows that $r_Y$ is also a topological isomorphism, so that $\pi^{\mathrm{top}}$ is a topological isomorphism. 
\end{proof}

We can now prove Theorem \ref{thrm:topological_version_main_thrm}. 

\begin{proof}[Proof of Theorem \ref{thrm:topological_version_main_thrm}]
    Let $(X, T)$ be a transitive topological $k$-step pro-nilsystem, and let 
    \[
    \pi \colon (X, T) \to (Y, S)
    \]
    be a topological factor map. 
    Since $(X, T)$ is distal and transitive, it is minimal and hence it is an inverse limit of minimal nilsystems. 
    Minimal nilsystems are uniquely ergodic, so $(X, T)$ is uniquely ergodic.
    Let $\mu$ be its unique invariant probability measure and define
    \[
    \nu := \pi_* \mu.
    \]
    Then $(X, \mu, T)$ is an ergodic measure-preserving $k$-step pro-nilsystem, and $\pi$ is a measure-theoretic factor map from $(X, \mu, T)$ to $(Y, \nu, S)$.

    By Theorem \ref{thm-main}, there exists an ergodic $k$-step pro-nilsystem $(\widetilde{Y}, \widetilde{\nu}, \widetilde{S})$ and a measure-theoretic isomorphism
    \[
    \Phi \colon (Y, \nu, S) \to (\widetilde{Y}, \widetilde{\nu}, \widetilde{S}).
    \]
    By the convention above, we assume that $(\widetilde{Y}, \widetilde{S})$ is a uniquely ergodic topological $k$-step pro-nilsystem with unique invariant measure $\widetilde{\nu}$.

    Then 
    \[
    \Phi \circ \pi \colon (X, \mu, T) \to (\widetilde{Y}, \widetilde{\nu}, \widetilde{S}) 
    \]
    is a measure-theoretic factor map between ergodic $k$-step pro-nilsystems. 
    By Proposition \ref{prop:factors_between_pronilsystems}, $\Phi \circ \pi$ agrees almost everywhere with a topological factor map 
    \[
    \widetilde{\pi} \colon (X, T) \to (\widetilde{Y}, \widetilde{S}).
    \]
    Consider the closed $S \times \widetilde{S}$-invariant subset 
    \[
    Z := (\pi, \widetilde{\pi})(X) \subseteq Y \times \widetilde{Y}.
    \]
    Since $(X, T)$ is distal and uniquely ergodic, its factor $(Z, S \times \widetilde{S})$ is also distal and uniquely ergodic. 

    Let 
    \[
    \lambda := (\pi, \widetilde{\pi})_* \mu
    \]
    be the unique invariant probability measure on $Z$. 
    Let 
    \[
    p \colon Z \to Y, \qquad \widetilde{p} \colon Z \to \widetilde{Y}
    \]
    be the coordinate projections.
    Then $p$ and $\widetilde{p}$ are both topological factor maps. 
    Moreover, since $\widetilde{\pi} = \Phi \circ \pi$ holds $\mu$-almost everywhere, we have that 
    \[
    \lambda = (\pi, \widetilde{\pi})_* \mu = (\Id_Y, \Phi)_* \nu.
    \]
    Thus $\lambda$ is the graph joining of the measure-theoretic isomorphism $\Phi$.
    It follows that both 
    \[
    p \colon (Z, \lambda, S \times \widetilde{S}) \to (Y, \nu, S), \qquad \widetilde{p} \colon (Z, \lambda, S \times \widetilde{S}) \to (\widetilde{Y}, \widetilde{\nu}, \widetilde{S})
    \]
    are measure-theoretic isomorphisms. 

    Applying Lemma \ref{lem:upgrade_mble_to_top_isom}, we get that $p$ and $\widetilde{p}$ are topological isomorphisms. 
    Therefore 
    \[
    \widetilde{p} \circ p^{-1} \colon (Y, S) \to (\widetilde{Y}, \widetilde{S})
    \]
    is a topological isomorphism. 
    We conclude that $(Y, S)$ is a topological $k$-step pro-nilsystem.
\end{proof}

\end{document}